\documentclass[11pt]{article}

\usepackage{amssymb,amsmath,amsfonts,amsthm}
\usepackage{latexsym}
\usepackage{graphics}
\usepackage{indentfirst}

\newtheorem*{thm*}{Theorem}
\newtheorem{thm}{Theorem}
\newtheorem{lem}[thm]{Lemma}

\newtheorem{obs}[thm]{Observation}
\newtheorem{cor}[thm]{Corollary}

\newtheorem{ques}[thm]{Question}

\newcommand{\N}{\mathbb{N}}

\begin{document}

\title{An Algebraic Approach for Counting DP-3-colorings of Sparse Graphs
}

\author{Samantha L. Dahlberg$^1$, Hemanshu Kaul$^2$, and Jeffrey A. Mudrock$^3$}

\footnotetext[1]{Department of Applied Mathematics, Illinois Institute of Technology, Chicago, IL 60616.  E-mail:  {\tt {sdahlberg@iit.edu}}}
\footnotetext[2]{Department of Applied Mathematics, Illinois Institute of Technology, Chicago, IL 60616.  E-mail:  {\tt {kaul@iit.edu}}}
\footnotetext[3]{Department of Mathematics, College of Lake County, Grayslake, IL 60030.  E-mail:  {\tt {jmudrock@clcillinois.edu}}}

\maketitle

\begin{abstract}
DP-coloring (or correspondence coloring) is a generalization of list coloring that has been widely studied since its introduction by Dvo\v{r}\'{a}k and Postle in 2015. As the analogue of the chromatic polynomial of a graph $G$, $P(G,m)$, and the list color function, $P_{\ell}(G,m)$, the DP color function of $G$, denoted by $P_{DP}(G,m)$, counts the minimum number of DP-colorings over all possible $m$-fold covers. It follows that $P_{DP}(G,m) \le P_{\ell}(G,m) \le P(G,m)$.
A function $f$ is chromatic-adherent if for every graph $G$, $f(G,a) = P(G,a)$ for some $a \geq \chi(G)$ implies that $f(G,m) = P(G,m)$ for all $m \geq a$.  It is known that the DP color function is not chromatic-adherent, but there are only two known graphs that demonstrate this.  
Suppose $G$ is an $n$-vertex graph and $\mathcal{H}$ is a 3-fold cover of $G$, in this paper we associate with $\mathcal{H}$ a polynomial $f_{G, \mathcal{H}} \in \mathbb{F}_3[x_1, \ldots, x_n]$ so that the number of non-zeros of $f_{G, \mathcal{H}}$ equals the number of $\mathcal{H}$-colorings of $G$.  We then use a well-known result of Alon and F\"{u}redi on the number of non-zeros of a polynomial to establish a non-trivial lower bound on $P_{DP}(G,3)$ when $2n > |E(G)|$. An easy consequence of this is that $P_{DP}(G, 3) \geq 3^{n/6}$ for every $n$-vertex planar graph $G$ of girth at least 5, improving the previously known bounds on both $P_{DP}(G, 3)$ and $P_{\ell}(G, 3)$. Finally, we use this bound to show that there are infinitely many graphs that demonstrate the non-chromatic-adherence of the DP color function.

\medskip
  


\end{abstract}

\section{Introduction}\label{intro}

In this paper all graphs are nonempty, finite, undirected loopless multigraphs.  For the purposes of this paper, a simple graph is a multigraph without any parallel edges between vertices.  Generally speaking we follow West~\cite{W01} for terminology and notation.  The set of natural numbers is $\N = \{1,2,3, \ldots \}$.  For $m \in \N$, we write $[m]$ for the set $\{1, \ldots, m \}$.  If $G$ is a graph and $S, U \subseteq V(G)$, we use $G[S]$ for the subgraph of $G$ induced by $S$, and we use $E_G(S, U)$ for the set consisting of all the edges in $E(G)$ that have one endpoint in $S$ and the other in $U$. When $u,v \in V(G)$ we use $E_G(u,v)$ to denote the set of edges in $E(G)$ with endpoints $u$ and $v$ (note $E_G(u,v) = E_G(v,u)$), and we let $e_G(u,v)$ denote the number of elements in $E_G(u,v)$.  When $G$ is a multigraph, the \emph{underlying graph of $G$} is the simple graph formed by deleting all parallel edges of $G$.  When $G$ is a simple graph, we can refer to edges by their endpoints; for example, if $u$ and $v$ are adjacent in the simple graph $G$, $uv$ or $vu$ refers to the edge between $u$ and $v$.  We say that a multigraph $G$ is $k$-degenerate if every subgraph of $G$ has a vertex of degree at most $k$.  If $G$ and $H$ are vertex disjoint multigraphs, we write $G \vee H$ for the join of $G$ and $H$. 

\subsection{DP-Coloring}

In classical vertex coloring we wish to color the vertices of a graph $G$ with up to $m$ colors from $[m]$ so that adjacent vertices receive different colors, a so-called \emph{proper $m$-coloring}. The smallest $k$ for which a proper $k$-coloring of $G$ exists is called the \emph{chromatic number} of $G$, and it is denoted $\chi(G)$.  List coloring is a well-known variation on classical vertex coloring that was introduced independently by Vizing~\cite{V76} and Erd\H{o}s, Rubin, and Taylor~\cite{ET79} in the 1970s.  For list coloring, we associate a \emph{list assignment} $L$ with a graph $G$ such that each vertex $v \in V(G)$ is assigned a list of colors $L(v)$.  Then, $G$ is \emph{$L$-colorable} if there is a proper coloring $f$ of $G$ such that $f(v) \in L(v)$ for each $v \in V(G)$ (we say $f$ is a \emph{proper $L$-coloring} of $G$).  A list assignment $L$ is called an \emph{$m$-assignment} for $G$ if $|L(v)|=m$ for each $v \in V(G)$.  We say $G$ is \emph{$m$-choosable} if $G$ is $L$-colorable whenever $L$ is an $m$-assignment for $G$.  

In 2015, Dvo\v{r}\'{a}k and Postle~\cite{DP15} introduced a generalization of list coloring called DP-coloring (they called it correspondence coloring) in order to prove that every planar graph without cycles of lengths 4 to 8 is 3-choosable. DP-coloring has been extensively studied over the past 7 years (see e.g.,~\cite{B17, BK182, KM20, KO18, LLYY19, M18}). Intuitively, DP-coloring is a variation on list coloring where each vertex in the graph still gets a list of colors, but identification of which colors are different can change from edge to edge.  Due to this property, DP-coloring multigraphs is not as simple as coloring the corresponding underlying graph (see~\cite{BKP17}).  Following~\cite{M22}, we now give the formal definition.  Suppose $G$ is a multigraph.  A \emph{cover} of $G$ is a triple $\mathcal{H} = (L,H,M)$ where $L$ is a function that assigns to each $v \in V(G)$ a set $L(v)=\{(v,a) : a \in A_v \}$ where $A_v$ is some nonempty finite set, $H$ is a multigraph with vertex set $\bigcup_{v \in V(G)} L(v)$, and $M$ is a function that assigns to each $e \in E(G)$ a matching $M(e)$ with the property that each edge in $M(e)$ has one endpoint in $L(u)$ and the other endpoint in $L(v)$ where $u$ and $v$ are the endpoints of $e$.  Moreover, $L$, $H$, and $M$ satisfy the following conditions: 

\vspace{5mm}

\noindent (1) For every $u \in V(G)$, $H[L(u)]=K_{|L(u)|}$; \\
(2)  For distinct edges $e_1, e_2 \in E(G)$, $M(e_1) \cap M(e_2) = \emptyset$; \\
(3)  For distinct vertices $u, v \in V(G)$, the set of edges between $L(u)$ and $L(v)$ in $H$ is $\bigcup_{e \in E_G(u,v)} M(e)$. 

\vspace{5mm}

Note that by conditions (2) and (3) in the above definition $H$ may contain parallel edges.  Furthermore, note that if $G$ is a simple graph, $H$ must be simple.  

Suppose $\mathcal{H} = (L,H,M)$ is a cover of $G$.  An \emph{$\mathcal{H}$-coloring} of $G$ is an independent set in $H$ of size $|V(G)|$.  It is immediately clear that an independent set $I \subseteq V(H)$ is an $\mathcal{H}$-coloring of $G$ if and only if $|I \cap L(u)|=1$ for each $u \in V(G)$.  We say $\mathcal{H}$ is \emph{$m$-fold} if $|L(u)|=m$ for each $u \in V(G)$.  Moreover, we say that $\mathcal{H}$ is a \emph{full $m$-fold cover} of $G$ if $|E_H(L(u),L(v))|= e_G(u,v) m$ whenever $u$ and $v$ are distinct vertices of $G$.  The \emph{DP-chromatic number} of $G$, denoted $\chi_{DP}(G)$, is the smallest $k$ such that an $\mathcal{H}$-coloring of $G$ exists whenever $\mathcal{H}$ is a $k$-fold cover of $G$.  Clearly, $\chi(G) \leq \chi_{DP}(G)$, and if $G$ is $d$-degenerate, then  $\chi_{DP}(G) \leq d+1$.

It is easy to demonstrate that DP-coloring is a generalization of list coloring.  Suppose that $K$ is an $m$-assignment for the simple graph $G$.  For each $v \in V(G)$, let $L(v) = \{ (v,j) : j \in K(v) \}$.  For each $uv \in E(G)$, let $M(uv) = \{(u,j)(v,j) : j \in L(u) \cap L(v) \}$.  Finally, let $H$ be the graph with vertex set $\bigcup_{v \in V(G)} L(v)$ and edge set that is the union of $\bigcup_{uv \in E(G)} M(uv)$ and $\left( \bigcup_{v \in V(G)} \bigcup_{i,j \in K(v), i \neq j} \{(v,i)(v,j)\} \right)$.  Now, let $\mathcal{H} = (L,H,M)$ and note that $\mathcal{H}$ is an $m$-fold cover of $G$.  Then, if $\mathcal{I}$ is the set of $\mathcal{H}$-colorings of $G$ and $\mathcal{C}$ is the set of proper $L$-colorings of $G$, the function $f: \mathcal{C} \rightarrow \mathcal{I}$ given by $f(c) = \{ (v, c(v)) : v \in V(G) \}$ is a bijection.  
 
\subsection{The DP Color Function}

In 1912 Birkhoff introduced the notion of the chromatic polynomial in hopes of using it to make progress on the four color problem.  For $m \in \N$, the \emph{chromatic polynomial} of a graph $G$, $P(G,m)$, is the number of proper $m$-colorings of $G$.  It can be shown that $P(G,m)$ is a polynomial in $m$ of degree $|V(G)|$ (see~\cite{B12}).  For example, $P(K_n,m) = \prod_{i=0}^{n-1} (m-i)$ and $P(T,m) = m(m-1)^{n-1}$ whenever $T$ is a tree on $n$ vertices (see~\cite{W01}).   

The notion of chromatic polynomial was extended to list coloring in the 1990s~\cite{AS90}.   In particular, if $L$ is a list assignment for $G$, we use $P(G,L)$ to denote the number of proper $L$-colorings of $G$. The \emph{list color function} $P_\ell(G,m)$ is the minimum value of $P(G,L)$ where the minimum is taken over all $m$-assignments $L$ for $G$.  It is clear that $P_\ell(G,m) \leq P(G,m)$ for each $m \in \N$ since we must consider the $m$-assignment that assigns $[m]$ to each vertex of $G$ when considering all possible $m$-assignments for $G$.  In general, the list color function can differ significantly from the chromatic polynomial for small values of $m$.  However, for large values of $m$, Dong and Zhang~\cite{DZ22} (improving upon results in~\cite{D92}, \cite{T09}, and~\cite{WQ17}) showed that for any graph $G$ with at least 4 edges, $P_{\ell}(G,m)=P(G,m)$ whenever $m \geq |E(G)|-1$.

With this in mind, we are ready to define a notion that will be important later in this paper.   Let $\mathcal{G}$ be the set of all finite multigraphs.  We say a function $f: \mathcal{G} \times \mathbb{N} \rightarrow \mathbb{N}$ is \emph{chromatic-adherent} if for every graph $G$,  $f(G,a) = P(G,a)$ for some $a \geq \chi(G)$ implies that $f(G,m) = P(G,m)$ for all $m \geq a$.  It is unknown whether the list color function is chromatic-adherent.

\begin{ques} [\cite{KN16}] \label{ques: funlist}
Is $P_{\ell}$ chromatic-adherent?
\end{ques} 

In 2019, the second and third author introduced a DP-coloring analogue of the chromatic polynomial called the DP color function in hopes of gaining a better understanding of DP-coloring and using it as a tool for making progress on some open questions related to the list color function~\cite{KM19}.  Since its introduction in 2019, the DP color function has received some attention in the literature (see e.g.,~\cite{BH21, DY21, HK21, KM21, LY22, MT20}). 

Suppose $\mathcal{H} = (L,H,M)$ is a cover of a multigraph $G$.  Let $P_{DP}(G, \mathcal{H})$ be the number of $\mathcal{H}$-colorings of $G$.  Then, the \emph{DP color function} of $G$, $P_{DP}(G,m)$, is the minimum value of $P_{DP}(G, \mathcal{H})$ where the minimum is taken over all full $m$-fold covers $\mathcal{H}$ of $G$.~\footnote{We take $\N$ to be the domain of the DP color function of any multigraph.  Also, we can restrict our attention to full $m$-fold covers since adding edges to a graph can't increase the number of independent sets of some prescribed size.}  It is easy to show that for any $m \in \N$, $$P_{DP}(G, m) \leq P_\ell(G,m) \leq P(G,m).$$ 

Finding lower bounds on both $P_{DP}(G,m)$ and $P_{\ell}(G,m)$ in the case that $G$ is planar and $m \in \{3,4,5\}$ has been considered in the literature (see e.g.,~\cite{BG22, S22, T07}).  In 2007, Thomassen~\cite{T07} gave an intricate proof showing that $P_{\ell}(G, 3) \geq 2^{n/10000}$ when $G$ is an $n$-vertex planar graph of girth at least 5.  This was recently improved as follows, and we will give a further  improvement below.

\begin{thm} [\cite{S22}] \label{thm: exponent}
If $G$ is an $n$-vertex planar graph with $n \geq 2$ and girth at least 5, then $P_{DP}(G, 3) \geq 2^{(n+890)/292}$.
\end{thm}

Unlike the list color function, it is known that $P_{DP}(G,m)$ does not necessarily equal $P(G,m)$ for sufficiently large $m$.  Indeed, Dong and Yang recently generalized a result of Kaul and Mudrock~\cite{KM19} and showed the following. 

\begin{thm} [\cite{DY21}] \label{thm: evencycle}
If $G$ is a simple graph that contains an edge $e$ such that the length of a shortest cycle containing $e$ is even, then there exists an $N \in \N$ such that $P_{DP}(G,m) < P(G,m)$ whenever $m \geq N$.  
\end{thm}

It was recently shown that the DP color function is not chromatic-adherent.   A \emph{Generalized Theta graph} $\Theta(l_1, \ldots, l_n)$ consists of a pair of end vertices joined by $n$ internally disjoint paths of lengths $l_1, \ldots, l_n \in \N$.  It is easy to see that $\chi_{DP}(\Theta(l_1, \ldots, l_n))=3$ whenever $n \geq 2$.

\begin{thm} [\cite{BK21}] \label{thm: negative}
If $G$ is $\Theta(2,3,3,3,2)$ or $\Theta(2,3,3,3,3,3,2,2)$, then $P_{DP}(G,3)=P(G,3)$ and there is an $N \in \N$ such that $P_{DP}(G,m) < P(G,m)$ for all $m \geq N$. 
\end{thm}

Wtih Theorem~\ref{thm: negative} in mind, the authors of~\cite{BK21} remarked that the two graphs in Theorem~\ref{thm: negative} are the only examples that they know of that demonstrate that the DP color function is not chromatic-adherent.  In this paper we show that there are infinitely many graphs with this property.

\subsection{Outline of Paper}

The proofs of our results are algebraic.  We begin by extending an algebraic technique for analyzing full 3-fold covers described in~\cite{KM20} to multigraphs.  This technique along with the following well-known result of Alon and F\"{u}redi will allow us to establish a non-trivial lower bound on $P_{DP}(G,3)$ for certain graphs $G$.

\begin{thm} [\cite{AF93}] \label{thm: AandF} 
Let $\mathbb{F}$ be an arbitrary field, let $A_1$, $A_2$, $\ldots$, $A_n$ be any non-empty subsets of $\mathbb{F}$, and let $B = \prod_{i=1}^n A_i$.  Suppose that $P \in \mathbb{F}[x_1, \ldots, x_n]$ is a polynomial of degree $d$ that does not vanish on all of $B$.  Then, the number of points in $B$ for which $P$ has a non-zero value is at least $\min \prod_{i=1}^n q_i$ where the minimum is taken over all integers $q_i$ such that $1 \leq q_i \leq |A_i|$ and $\sum_{i=1}^n q_i \geq -d + \sum_{i=1}^n |A_i|$.
\end{thm}

We prove the following.

\begin{thm} \label{thm: count3}
Suppose $G$ is a multigraph with $\chi_{DP}(G) \leq 3$.  Also, suppose that $|V(G)| = n$, $|E(G)|=l$, and $2n \geq l$.  Then, $P_{DP}(G, 3) \geq 3^{n - l/2}$.
\end{thm}

As an immediate application, consider a $n$-vertex planar graph $G$ of girth at least 5. It is known (see~\cite{DP15}) that $\chi_{DP}(G) \le 3$. Since number of edges in $G$ is at most $5n/3$, Theorem~\ref{thm: count3} implies the following.

\begin{cor}\label{cor:planar}
Let $G$ be a $n$-vertex planar graph of girth at least 5, then $P_{DP}(G, 3) \geq 3^{n/6}$.
\end{cor}

Corollary~\ref{cor:planar} improves Theorem~\ref{thm: exponent} whenever $n \geq 12$.  It was also observed in~\cite{BG22} that the same lower bound as Corollary~\ref{cor:planar} holds for $P_{\ell}(G, 3)$.

In Section~\ref{proofs}, we give the proof of Theorem~\ref{thm: count3}, demonstrate that its bound is tight, and in the process, prove the following.

\begin{thm} \label{thm: infinite}
There are infinitely many graphs $G$ for which $\chi_{DP}(G) = 3$, $P_{DP}(G,3) = P(G,3)$, and there is an $N_G \in \N$ such that $P_{DP}(G,m) < P(G,m)$ whenever $m \geq N_G$.
\end{thm}  

\section{Proofs of Results}\label{proofs}

From this point forward whenever $G$ is a multigraph with $n$ vertices, we suppose that $V(G) = \{v_1, \ldots, v_n\}$.  We will also suppose, unless otherwise noted, that all addition and multiplication is performed in $\mathbb{F}_3$ where $\mathbb{F}_3$ is the finite field of order 3.  Suppose $\mathcal{H} = (L,H,M)$ is a full 3-fold cover of the multigraph $G$ on $n$ vertices.  From this point forward, we will always assume under this set up that $L(v) = \{(v,j) : j \in \mathbb{F}_3 \}$.  Suppose $e$ is an arbitrary element of $E(G)$ with endpoints $v_i$ and $v_j$ where $i < j$.  Also, for each $k \in \mathbb{F}_3$, suppose the edge in $M(e)$ with endpoint $(v_i,k)$ has $(v_j,c_k)$ as its other endpoint. The \emph{permutation of $\mathcal{H}$ associated with $M(e)$}, denoted $\sigma_{e}^{\mathcal{H}}$, is the permutation $\sigma_e^{\mathcal{H}} : \mathbb{F}_3 \rightarrow \mathbb{F}_3$ given by $\sigma_e^{\mathcal{H}}(k)=c_k$.  We will now associate a polynomial in $\mathbb{F}_3[x_1, \ldots, x_n]$ with $\mathcal{H}$.  Before we do this, we need an observation that was specifically used for full 3-fold covers in~\cite{KM20}.

\begin{obs} \label{obs: crucial}
Suppose $\sigma$ is a permutation of $\mathbb{F}_3$.  Then, either $z - \sigma(z)$ is the same for all $z \in \mathbb{F}_3$, or $z + \sigma(z)$ is the same for all $z \in \mathbb{F}_3$.
\end{obs}

Now, let the \emph{linear factor of $\mathcal{H}$ associated with $M(e)$}, denoted $l_e^{\mathcal{H}}(x_1, \ldots, x_n)$, be the polynomial in $\mathbb{F}_3[x_1, \ldots, x_n]$ given by $(x_i + (-1)^c x_j - a)$ where $c$ and $a$ are chosen so that $(x_i + (-1)^c x_j - a)$ is zero if and only if $x_i, x_j$ are elements of $\mathbb{F}_3$ satisfying $\sigma_{e}^{\mathcal{H}}(x_i) = x_j$.  Notice that Observation~\ref{obs: crucial} guarantees that such a $c$ and $a$ must exist~\footnote{It is also not too difficult to prove that $c$ and $a$ are unique by considering each possible matching.}.  Finally, we let the \emph{graph polynomial of $G$ associated with $\mathcal{H}$}, denoted $f_{G, \mathcal{H}}(x_1, \ldots, x_n)$, be the polynomial in $\mathbb{F}_3[x_1, \ldots, x_n]$ given by
$$f_{G, \mathcal{H}}(x_1, \ldots, x_n) =  \prod_{E_G(v_i,v_j) \neq \emptyset, \; i<j} \left (\prod_{e \in E_G(v_i,v_j)} l_e^{\mathcal{H}}(x_1, \ldots, x_n) \right).$$   
Notice that $f_{G, \mathcal{H}}$ is a polynomial of degree $|E(G)|$.  Also, by construction, the following observation is immediate.

\begin{obs} \label{obs: count}
Let $\mathcal{C} = \{ C \subset V(H): \text{$|C \cap L(v)|=1$ for each $v \in V(G)$} \}$, and note that all $\mathcal{H}$-colorings of $G$ are contained in $\mathcal{C}$.  Suppose $C \in \mathcal{C}$ and $C = \{(v_1, c_1), \ldots, (v_n,c_n) \}$.  Then, $C$ is an $\mathcal{H}$-coloring of $G$ if and only if $f_{G, \mathcal{H}}(c_1, \ldots, c_n) \neq 0$.  Consequently, if $B = \prod_{i=1}^n \mathbb{F}_3$, then the number of points in $B$ for which $f_{G, \mathcal{H}}$ has a non-zero value is $P_{DP}(G, \mathcal{H})$.  
\end{obs}

It is easier to apply the following corollary of Theorem~\ref{thm: AandF}.

\begin{cor} [\cite{BG22}] \label{thm: bound}
Let $\mathbb{F}$ be an arbitrary field, let $A_1$, $A_2$, $\ldots$, $A_n$ be any non-empty subsets of $\mathbb{F}$, and let $B = \prod_{i=1}^n A_i$.  Suppose that $P \in \mathbb{F}[x_1, \ldots, x_n]$ is a polynomial of degree $d$ that does not vanish on all of $B$.  If $S = \sum_{i=1}^n |A_i|$, $t = \max |A_i|$, $S \geq n + d$, and $t \geq 2$, then the number of points in $B$ for which $P$ has a non-zero value is at least $t^{(S-n-d)/(t-1)}.$
\end{cor} 

We are now ready to complete the proof of Theorem~\ref{thm: count3}.

Suppose $V(G) = \{v_1, \ldots, v_n \}$ and $\mathcal{H} = (L,H,M)$ is a full 3-fold cover of $G$ with $P_{DP}(G, \mathcal{H}) = P_{DP}(G,3)$.  Suppose also that the vertices of $H$ are arbitrarily named so that $L(v) = \{(v,j) : j \in \mathbb{F}_3 \}$ for each $v \in V(G)$.  Let $A_i = \mathbb{F}_3$ for each $i \in [n]$, and let $B = \prod_{i=1}^n A_i$.  Note that $f_{G, \mathcal{H}}(x_1, \ldots, x_n) \in \mathbb{F}_3[x_1, \ldots, x_n]$ has degree $l$.  Moreover, $f_{G, \mathcal{H}}(x_1, \ldots, x_n)$ doesn't vanish on all of $B$ by the fact that $\chi_{DP}(G) \leq 3$ and Observation~\ref{obs: count}.  

Thus, the number of points in $B$ for which $f_{G, \mathcal{H}}(x_1, \ldots, x_n)$ has a non-zero value is at least $3^{(3n-n-l)/(3-1)}$ by Corollary~\ref{thm: bound}.  The theorem then follows from Observation~\ref{obs: count}.

We now mention three simple examples that demonstrate the tightness of Theorem~\ref{thm: count3}.  First, notice that if $G_1$ is an edgeless graph on $n$ vertices, then $P_{DP}(G_1,3) = 3^n$.  Second, suppose that $G_2$ is a multigraph on 2 vertices with 2 edges.  Then, it easy to see that $P_{DP}(G_2,3) = 3$ (see Proposition~11 in~\cite{M22}).  Finally, suppose that $k \in \N$ and $G_3 = K_1 \vee C_{2k+2}$.  Notice that $|V(G_3)| = 2k+3$ and $|E(G_3)| = 4k+4$.  It is shown in~\cite{BH21} that $P_{DP}(G_3,3)=3$.  It is also worth mentioning that Theorem~\ref{thm: count3} tells us that $P_{DP}(G_3,3) \geq 3$ in a manner that is much more elegant than the result used in~\cite{BH21} to demonstrate the same lower bound (see Lemma~16 in~\cite{BH21}).  

We now turn our attention to proving Theorem~\ref{thm: infinite}.  First, we need a definition and a result.  A graph $G$ is said to be \emph{uniquely $k$-colorable} if there is only one partition of its vertex set into $k$ independent sets.  It is well known (see~\cite{T81}) that if $G$ is a uniquely $k$-colorable graph with $n$ vertices, then $|E(G)| \geq (k-1)n - k(k-1)/2$.  It is also easy to observe that when $G$ is uniquely $k$-colorable, $\chi(G) = k$ and $P(G,k) = k!$.  With this in mind, we have the following lemma.

\begin{lem} \label{lem: unique}
Suppose $G$ is a uniquely $3$-colorable graph on $n$ vertices with $\chi_{DP}(G)=3$.  If $|E(G)| = 2n - 3$, then $P_{DP}(G,3)=P(G,3)$.
\end{lem}

\begin{proof}
Clearly, $P(G,3) = 6$ and $P_{DP}(G,3) \leq P(G,3)$.  Since $\chi_{DP}(G)=3$, Theorem~\ref{thm: count3} implies that $P_{DP}(G,3) \geq 3^{n - (2n-3)/2} = 3^{3/2}$.  Since $P_{DP}(G,3)$ is an integer, this means $P_{DP}(G,3) \geq 6$.
\end{proof}

We are now ready to prove Theorem~\ref{thm: infinite}

\begin{proof}
For each $k \in \N$, let $H_k$ be the graph obtained from a copy of $K_1 \vee P_{2k+2}$ by adding a new vertex $z$ and then adding an edge between $z$ and each endpoint of the copy of $P_{2k+2}$.  To prove Theorem~\ref{thm: infinite}, we will show that for each $k \in \N$, $P_{DP}(H_k,3)=P(H_k,3)$, and there is an $N_{H_k} \in \N$ such that $P_{DP}(H_k,m) < P(H_k,m)$ whenever $m \geq N_{H_k}$.

Now, fix a $k \in \N$.  It is easy to see that $H_k$ is both 2-degenerate and uniquely 3-colorable.  Consequently, $\chi_{DP}(H_k) = 3$.  Moreover, $|V(H_k)| = 2k+4$ and $|E(H_k)| = 4k+5 = 2(2k+4)-3$.  So, Lemma~\ref{lem: unique} implies that $P_{DP}(H_k,3)=P(H_k,3)$.

Finally, notice that each of the two edges in $H_k$ incident to $z$ have the property that the smallest cycle in $H_k$ containing the edge is of length 4.  So, Theorem~\ref{thm: evencycle}
implies there is an $N_{H_k} \in \N$ such that $P_{DP}(H_k,m) < P(H_k,m)$ whenever $m \geq N_{H_k}$.    
\end{proof}

\end{document}